\documentclass{comsoc2008}

\title{Dodgson's Rule\\
Approximations and Absurdity}
\author{John C. M\textsuperscript{c}Cabe-Dansted\footnotemark[1]
}

\usepackage{prettyref}
\usepackage{nicefrac}
\usepackage{color}
\usepackage{graphicx}
\usepackage{amssymb}
\usepackage{amsmath}
\usepackage{amsthm}
\usepackage[authoryear]{natbib}

\newcommand{\noun}[1]{\textsc{#1}}
\providecommand{\tabularnewline}{\\}
  \newtheorem{thm}{Theorem}[section]
  \newtheorem{defn}[thm]{Definition}
  \newtheorem{lem}[thm]{Lemma}
  \newtheorem{cor}[thm]{Corollary}
  \newtheorem{note}[thm]{Note}
  \newtheorem{example}[thm]{Example}

\usepackage{natbib}
\usepackage{graphics}
\usepackage{url}
\usepackage{nicefrac}

\providecommand\ifverbose[1]{#1}

\newtheorem{tbl}{Table}
\newcommand{\BeginTable}{\begin{tbl}}
\newcommand{\EndTable}{\end{tbl}}

 \newcommand{\funcc}[1]{\text{#1}}

\newcommand\adv{\funcc{adv}}
\def\var{\func{var}}
\def\cov{\func{cov}}

\newcommand\sups[1]{\textsuperscript{#1}}

\def\cal{\mathcal}

\newcommand{\LL}{\mbox{$\cal L$}}
\newcommand{\NN}{\mbox{$\cal N$}}

\renewcommand{\underbar}[1]{\textbf{#1}\index{#1}}

\def\tth{\textsuperscript{th}\hspace{0.2ex}}

\usepackage{hyperref}

\newcommand{\iflong}[1]{#1}
\newcommand{\ifshort}[1]{}
\newcommand{\ifpolya}[1]{#1}
\newcommand{\ifnotpolya}[1]{}

\usepackage{url}
\usepackage{hyperref}

\makeatother

\begin{document}
\newcommand{\rulescore}[1]{\textrm{Sc}_{\mathbf{#1}}}
\newcommand{\ScS}{\rulescore{S}}
\newcommand{\ScT}{\rulescore{T}}
\newcommand{\ScQ}{\rulescore{Q}}
\newcommand{\ScC}{\rulescore{C}}
\newcommand{\ScR}{\rulescore{R}}
\newcommand{\ScRR}{\rulescore{\&}}
\newcommand{\ScD}{\rulescore{D}}
\newcommand{\eqgap}{\;\,\,}
\newcommand{\jnicefrac}[2]{\nicefrac{#1}{#2}}
\newcommand{\jlog}{\ln}
\newcommand{\outcomes}{\mathcal{U}}
\newcommand{\events}{\mathcal{F}}
\newcommand{\ceil}[1]{\left\lceil {#1} \right\rceil }
\newcommand{\func}[1]{\text{#1}}
\newcommand{\floor}[1]{\left\lfloor #1\right\rfloor }
\newcommand{\jseq}[2]{#1_{1},#1_{2},\ldots,#1_{#2}}
\newcommand{\term}[1]{\textbf{#1}\index{#1}}
\newcommand{\probmeas}{m}
\newcommand{\probspace}[1]{(\outcomes_{#1},\probmeas_{#1})}
\newcommand{\even}{2\mathbb{Z}}
\newcommand{\integers}{\mathbb{Z}}
\newcommand{\union}{\cup}
\newcommand{\allorders}{\LL(\allalts)}
\newcommand{\allalts}{\mathcal{A}}
\newcommand{\altalts}{\mathcal{C}}
\newcommand{\reals}{\mathbb{R}}
\newcommand{\vctr}[1]{\mathbf{#1}}
\newcommand{\profile}{\mathcal{P}}
\newcommand{\altprofile}{\mathcal{Q}}
\newcommand{\fQ}{F}

\newcommand{\linord}{\mathbf{v}}
\newcommand{\ith}{i^{\textrm{th}}}
\newcommand{\TF}{W}
\newcommand{\MF}{W^{\profile}}
\newcommand{\bigO}{\mathcal{O}}

\newcommand{\jE}[1]{E[#1]}
\newcommand{\jcov}[2]{\cov(#1,#2)}
\newcommand{\jvar}[1]{\var(#1)}
\newcommand{\allsubords}[1]{Z(#1)}
\newcommand{\allmultisets}[1]{M(#1)}
\newcommand{\votsit}{U}
\newcommand{\votswapped}{\tilde{\profileswapped}}
\newcommand{\profileswapped}{V}
\newcommand{\swapset}{\setofswaps}
\newcommand{\altswapset}{\mathfrak{T}}
\newcommand{\alinord}{\vctr{v}}
\newcommand{\DS}{Sc_{\textbf{d}}}
\newcommand{\numswaps}{\Sigma\setofswaps}
\newcommand{\altlinord}{\vctr{w}}
\newcommand{\jzseq}[2]{#1_{0},#1_{1},\ldots,#1_{#2}}
\newcommand{\abv}[2]{\mathfrak{A}_{#1}\left(#2\right)}
\newcommand{\belo}[2]{\mathfrak{B}_{#1}\left(#2\right)}
 \newcommand{\klength}{|\vctr{c}|}
 \newcommand{\kstring}{c_{1}\cdots c_{\klength}}
\newcommand{\kkstring}{c_{1}c_{2}\cdots c_{\klength-1}}
 \newcommand{\alength}{|\vctr{a}|}
\newcommand{\astring}{a_{1}\cdots a_{\alength}}
 \newcommand{\aastring}{a_{1}\cdots a_{\alength-1}}
\newcommand{\jU}{T_{V}}
\newcommand{\jUmax}{T_{U}}
\newcommand{\setofswaps}{\mathfrak{S}}
\newcommand{\jS}{S_{\profile}}
\newcommand{\jT}{S_{\profileswapped}}
\newcommand{\LPmat}{A}
\newcommand{\LPbch}{b}
\newcommand{\LPfch}{c}
\newcommand{\LPb}{\vctr{\LPbch}}
\newcommand{\LPf}{\vctr{f}}
\newcommand{\LPm}{M}
\newcommand{\LPn}{N}
\newcommand{\LPbb}{(\jseq{\LPbch}{\LPm})^{T}}
\newcommand{\LPff}{(\jseq{\LPfch}{\LPm})^{T}}
\newcommand{\LPwhole}{(\LPm,\LPn,\LPmat,\LPb,\LPf)}
\newcommand{\vone}[1]{\vctr{1}_{#1}}
\newcommand{\rankdist}[2]{\func{dist}(#1,#2)}
\newcommand{\rankdistab}[4]{\func{dist}_{#3,#4}(#1,#2)}
\newcommand{\casesof}[2]{#1^{#2}}
\newcommand{\posreals}{\reals^{+}}

\newcommand{\intrange}[3]{\{#1,#2,\ldots,#3\}}
\newcommand{\intrangeone}[1]{\intrange{1}{2}{#1}}
\newcommand{\nsit}{\tilde{n}}
\newcommand{\dodgsonalt}{a}
\newcommand{\dodgalt}{a}
\newcommand{\dqalt}{d}
\newcommand{\allords}{\mathcal{L}\left(\mathcal{A}\right)}
\newcommand{\antiordv}{\tilde{\alinord}}
\newcommand{\votset}{\mathcal{S}}
\newcommand{\votcount}[2]{\#_{#1}(#2)}
\newcommand{\SSS}{\mbox{${\cal S}$}}
\newcommand{\altvotsit}{V}

\begin{abstract}
With the Dodgson rule, cloning the electorate can change the winner,
which \citet{Yo77} considers an {}``absurdity''. Removing
this absurdity results in a new rule \citep{fish77} for which we can compute
the winner in polynomial time \citep{rospvo03}, unlike the traditional Dodgson rule.
We call this rule DC and introduce two new related rules (DR and D\&). Dodgson did not explicitly
propose the {}``Dodgson rule'' \citep{tideman1}; we argue that
DC and DR are better realizations of the principle behind the Dodgson
rule than the traditional Dodgson rule. These rules, especially D\&,
are also effective approximations to the traditional Dodgson's rule.
We show that, unlike the rules we have considered previously,
the DC, DR and D\& scores differ from the Dodgson score by no more than a fixed amount given a fixed number of alternatives, and thus these new rules converge to Dodgson under any reasonable assumption on voter behaviour, including the Impartial Anonymous Culture assumption.
{\normalsize }{\normalsize \par}
\end{abstract}

\section{Introduction}

\footnotetext[1]{The author would like to thank Arkadii Slinko for his many valuable suggestions and references.}
 Finding the Dodgson winner to an election can be very difficult,
\citet{batotr89} proved that determining whether an alternative is
the Dodgson winner is an NP-hard problem. Later \citet{hehero97}
refined this result by showing that the Dodgson winner problem was
complete for parallel access to NP, and hence not in NP unless the polynomial hierarchy collapses\@. This result was of interest
to computer science as the previously known problems in this complexity
class were obscure by comparison.

For real world elections we do not want intractable problems. \citet{tideman1}
proposed a simple rule to approximate the Dodgson rule. The impartial
culture assumption states that all votes are independent and equally
likely. Under this assumption, it has been proven that the probability
that Tideman's rule picks the Dodgson winner converges to one as the
number of voters goes to infinity \citep{DaPrSl07}. Our paper also
showed that this growth was not exponentially fast. Two rules
have been independently proposed for which this convergence is exponentially
fast, our Dodgson Quick (DQ) rule and the \texttt{GreedyWinner} algorithm
proposed by \citet{homan-2005-}. It is usually easy to verify that
the DQ winner and GreedyWinner are the same as the Dodgson winner, a property
that \citeauthor{homan-2005-} formalise as being {}``frequently
self-knowingly correct''. However the proofs of convergence 
depended heavily on the unrealistic impartial culture assumption.
We shall show that under the Impartial Anonymous Culture (IAC) these rules do not converge. 

The importance of ensuring that statistical results hold on reasonable
assumptions on voter behaviour is considered by \citet{PrRo07}. They
define {}``deterministic heuristic polynomial time algorithm'' in
terms of both the problem to be solved and a probability distribution
over inputs. However they do not consider the issue of whether a heuristic
is self-knowingly correct. Thus we extend the concept of an algorithm
being {}``frequently self-knowingly correct'' to allow particular
probability distributions to be specified. \citet{PrRo07} also propose
{}``Junta'' distributions; these distributions are intended to produce
problems that are harder than would be produced under reasonable assumptions
on voter behaviour. Hence if it is easy to solve a problem when input
is generated according to a Junta distribution it is safe to assume
that it will be easy under any reasonable assumption of voter behaviour.
We will not use Junta distributions, but instead simply use that fact
that even {}``neck-and-neck'' national elections are won by thousands
of votes. We will also show that the rules we have considered
previously, Tideman, Dodgson Quick etc., do not converge to Dodgson's
rule under \iflong{the} IAC.

However, the reason that the Dodgson rule is so hard to compute is
because cloning the electorate can change the winner. That is, if we
replace each vote with two (or more) identical votes, this may
change the winner. When discussing majority voting \citet{Yo77} described
this property as an {}``absurdity''. Young suggested that such absurdities
be fixed in majority voting by allowing fractions of a vote to be
deleted. \citet{fish77} proposed a similar modification to the traditional Dodgson
rule, which we call Dodgson Clone. The Dodgson Clone scores can be
computed by relaxing the integer constraints on the Integer Linear
Program that \citet{batotr89} proposed to calculate the Dodgson score;
normal (rational) Linear Programs can be solved in polynomial
time, we can compute the Dodgson Clone score in polynomial time \citep{rospvo03}.

The Dodgson Clone rule is also an effective approximation to the Dodgson
rule.  In computer science, an approximation typically refers to an algorithm that selects a value that is always accurate to within some error. This form of approximation is not meaningful when selecting a winner, although these rules can be used to approximate the frequency that Dodgson winner has some property. For example, \citet{shah1} used Tideman's rule to approximate the frequency that the Dodgson winner matched the winners according to other rules, and so such Tideman, DQ etc.\ can be considered approximations of Dodgson's rule in a loose sense. However we can approximate the Dodgson score. We will show that for a fixed number of alternatives, the Dodgson Clone score approximates the Dodgson score to within a constant error.
 As it is implausible that the margin by which the winner wins the election will not grow with the size of the electorate, the Dodgson winner will converge to the Dodgson Clone winner under any reasonable assumption of voter behaviour. In particular we will show that they will converge under 
the Impartial Anonymous Culture assumption.

We propose two closer approximations Dodgson Relaxed (DR) and Dodgson
Relaxed and Rounded (D\&). These approximations, like the traditional
Dodgson rule, are not resistant to cloning the electorate. This allows
them to be closer to the Dodgson rule than Dodgson Clone, both rules
converge to the Dodgson rule exponentially quickly under the Impartial
Culture assumption. The DR rule is superior to the Dodgson rule in
the sense that it can split ties in favour of alternatives that are
fractionally better. The D\& scores are rounded up, so the D\& rule
does not have this advantage. However it is exceptionally close
to Dodgson. In 43 million elections randomly generated
according to various assumptions on voter behaviour, the D\& winner
differed from the Dodgson winner in only one election.

The approximation proposed by \citet{RePEc:huj:dispap:dp466} is similar
to these approximations in the sense that it involves a relaxation
of the integer constraints. However their approximation is randomised,
and thus quite different from our deterministic approximations. Using
a randomised approximation as a voting rule would be unusual, and
they do not discuss the merits of such a rule. Thus the focus of their
paper is quite different, as we present rules that we argue are superior
to the traditional formalisation of the Dodgson rule. 
  Additionally, they do not discuss the issue of frequently self-knowing correctness.

Another approach to computing the Dodgson score has been to limit
some parameter. \citet{batotr89} showed that computing the Dodgson
scores and winner is polynomial when either the number of voters or
alternatives is limited. It was shown that computing these from a
voting situation is logarithmic with respect to the number of voters
when the number of alternatives is fixed \citep{Da06}, and hence
Dodgson winner is Fixed Parameter Tractable (FPT) with number of alternatives
as the fixed parameter. It is now also known that the Dodgson winner
is FPT when the Dodgson score is taken as the fixed parameter \citep{BeGuNi08}.

Thus we will define the Dodgson based rules in terms of Condorcet-tie
winners, rather than Condorcet winners. As we will discuss briefly,
this does not affect convergence.

\section{\label{sub:Advantages-based-Scores}Preliminaries}

In our results we use the term agent in place of voter and alternative
in place of candidate, as not all elections are humans voting other
humans into office. For example, in direct democracy, the citizens
vote for laws rather than candidates.

\ifverbose{We assume that agents' preferences are transitive, i.e.\ if
they prefer $a$ to $b$ and prefer $b$ to $c$ they also prefer
$a$ to $c$. We also assume that agents' preferences are strict,
if $a$ and $b$ are distinct they either prefer $a$ to $b$ or $b$
to $a$. Thus we may consider each agent's preferences to be a ranking
of each alternative from best to worst.}

Let $\allalts$ and $\NN$ be two finite sets of cardinality $m$
and $n$ respectively. The elements of $\allalts$ will be called
alternatives, the elements of $\NN$ agents. We represent a \underbar{vote}
by a linear order of the $m$ alternatives. We define a \underbar{profile}
to be an array of $n$ votes, one for each agent. Let $\profile=(P_{1},P_{2},\ldots,P_{n})$
be our profile. If a linear order $P_{i}\in\LL(A)$ represents the
preferences of the $i$\tth agent, then by $aP_{i}b$, where $a,b\in\allalts$,
we denote that this agent prefers $a$ to $b$.

A multi-set of linear orders of $\allalts$ is called a \underbar{\label{voting-situation}voting
situation}. A voting situation specifies  which linear orders were
submitted and how many times they were submitted but not  who submitted
them. \iflong{A voting situation is sometimes referred to as a {}``succinct''
election, as it can be represented succinctly when there only a few
alternatives \citep{FaHeHe06}.} 

The Impartial Culture (IC) assumption is that each profile is equally likely. The Impartial Anonymous Culture (IAC) assumption is that each voting situation is equally likely. To understand the difference, consider a two alternative election with billions of agents; under IC it is almost certain that each alternative will get 50\% ($\pm 0.5\%$) of the vote; under IAC, 50.0\% is no more likely than any other value.

Let $\profile=(P_{1},P_{2},\ldots,P_{n})$ be our profile. We define
$n_{xy}$ to be the number of linear orders in $\profile$ that rank
$x$ above $y$, i.e. $n_{xy}\equiv\#\{i\mid xP_{i}y\}$. 

\begin{defn}
The \underbar{advantage} of $a$ over $b$ is defined as follows:\[
\text{adv}(a,b)=\max(0,n_{ab}-n_{ba})\]
A \underbar{Condorcet winner} is an alternative $a$ for which $\adv(a,b)>0$
for all other alternatives $b$.  We define a \underbar{Condorcet-tie winner}, to be an alternative
$a$ such $\adv(b,a)=0$ for all other alternatives $a$. A Condorcet winner or Condorcet-tie winner does not always
exist.
\end{defn}
It is traditional to define the Dodgson score of an alternative as
the terms of the minimum number of swaps of neighbouring alternatives
required to make that alternative \emph{defeat} all others in pairwise
elections, i.e. make the alternative a Condorcet winner. When not requiring solutions to be integer this becomes
undefined, as if we defeat an alternative by $\epsilon>0$ then there
exists a better solution where we defeat the alternative by only $\nicefrac{\epsilon}{2}$.

For this reason, when defining the Dodgson scores we only require
that the alternative defeat \emph{or tie} other alternatives, i.e. make the alternative a Condorcet-tie winner. For
better consistency with the more traditional Dodgson rule we could
define the Condorcet winner as an alternative $a$ for which $\adv(a,b)\geq1$.
However this would mean that the Dodgson Clone rule would not be resistant
to cloning of the electorate.

This difference in definition does not affect convergence. Our proof
of convergence relies only on fact that Dodgson, D\&, DR and DC scores
differ by at most a fixed amount $(\bigO(m!))$ when the number of
alternatives is fixed. To convert a Condorcet-tie winner $c$ into
a Condorcet winner $c$ we need to swap $c$ over at most $(m-1)$
alternatives, each requiring at most $(m-1)$ swaps of neighbouring
alternatives. Hence the difference between the score according to
these different definitions of Dodgson is at most $(m-1)^{2}$. 

We will now define a number of rules in terms of scores. The winner
of each rule below is the alternative with the lowest score.

The \textbf{Dodgson score}\index{C3 rules!Dodgson's rule}\index{Dodgson's rule}
(\citealt{dodgson}, see e.g. \citealt{black,tideman1}), which we
denote as $\ScD(a)$, of an alternative $a$ is %
defined
as the minimum number of swaps of neighboring alternatives required
 to make $a$ a Condorcet-tie winner.  We call the alternative(s)
with the lowest Dodgson score(s) the \textbf{Dodgson winner}(s). \citep{batotr89} 

\index{C2 rules!Simpson's rule}\index{Simpson's rule}

The \textbf{Tideman score} $\ScT(a)$ of an alternative $a$ is:
\begin{align*}
\ScQ(a)  = & \sum_{b\ne a}\text{adv}(b,a).\end{align*}
The \textbf{Dodgson Quick (DQ) score} $\ScQ(a)$, of an alternative
$a$ is \begin{align*}
\ScQ(a)  = & \sum_{b\ne a}\fQ(b,a),\text{ where }\fQ(b,a)=\left\lceil \frac{\text{adv}(b,a)}{2}\right\rceil .\end{align*}

Although the definitions of $\ScQ$ and $\ScT$ are very similar,
the Dodgson Quick rule converges exponentially fast to Dodgson's rule
under the Impartial Culture assumption, where as Tideman's rule does
not \citep{DaPrSl07}. This is because Dodgson and DQ are more sensitive
to a large number of alternatives defeating $a$ by a small odd margin (e.g. 1)
than Tideman is.

We define the $k$-Dodgson score $\ScD^{k}(d)$ of $a$ as being the
Dodgson score of $a$ in a profile where each agent has been replaced
with $k$ clones, divided by $k$. That is, where $\profile$ is our
fixed profile, $\profile^{k}$ is the profile with each agent replaced
with $k$ clones, and $\ScD[\profile](a)$ is the Dodgson score of
$a$ in the profile $\profile$, then \begin{align*}
\ScD^{k}(a) & =\ScD^{k}[\profile](a)=\frac{\ScD[\profile^{k}](a)}{k}\end{align*}

We define the \textbf{Dodgson Clone (DC) score} $\ScC(a)$ of an alternative
$a$ as $\min_{k}\ScD^{k}(a)$. The DC score can be equivalently defined
by modifying the Dodgson rule to allow votes to be split into rational
fractions and allowing swaps to be made on those fractions of a vote.
Note that like the Tideman approximation, the DC score is less sensitive
than Dodgson to a large number of alternatives defeating $a$ by a
margin of $1$, so the DC score is unlikely to converge to Dodgson
as quickly as DQ under the Impartial Culture assumption.

We define the \textbf{Dodgson Relaxed (DR) score} $\ScR$ as with
the Dodgson score, but allow votes to be split into rational fractions.
However we require that $a$ be swapped over $b$ at least $F(b,a)$
times. Thus $\ScR(d)\geq\ScQ(d)$ and the Dodgson Relaxed rule will
converge at least as quickly as DQ. The DR rules thus sacrifices independence
to cloning of the electorate to be closer to the Dodgson rule than
DC. 

The \textbf{Dodgson Relaxed and Rounded (D\&) score} $\rulescore{\&}$
is the DR score rounded up, i.e. $\rulescore{\&}(d)=\ceil{\ScR(d)}$.

\providecommand{\tth}{\sups{th}}

\section{\noun{Dodgson Linear Programmes}}
\ifshort{The Dodgson Clone scores can be computed by relaxing the integer constraints on the Integer Linear Programme (ILP) for Dodgson's rule \citep{rospvo03}.  In this section we will show that the difference between the solutions of the ILP and LP are $\bigO(m!)$. 
}
\iflong{
\citet{batotr89} defined an ILP for determining the Dodgson score
of a candidate as follows:\\
$\min\sum_{ij}jx_{ij}$ subject to\\
$\sum_{j}x_{ij}=N_{i}$ (for each type of vote $i$)\\
$\sum_{ij}e_{ijk}x_{ij}\geq d_{k}$ (for each alternative $k$)\\
$x_{ij}\geq0$, and each $x_{ij}$ must be integer,
where $x_{ij}$ are the variables and the candidate $d$ is swapped
up $j$ positions $i$ times in votes of type $i$. The constant $N_{i}$
represents the number of votes of type $i$. The constant $e_{ijk}$
is $1$ if moving $d$ up $j$ positions in vote type $i$ raises
$d$ above $k$, and 0 otherwise. The $d_{k}$ is the {}``minimum
number of votes'' that $d$ must gain to defeat $k$. This definition
opens a potential ambiguity: is $d_{k}=\adv(k,d)/2$ or is $d_{k}=F(k,d)=\ceil{\adv(k,d)/2}$.
This ambiguity does not affect the ILP as both definitions are equivalent
when we require that the variables be integer. However when we relax
the requirement that the variables be integer, we get the DC score
if $d_{k}=\adv(k,d)/2$ and the DR score if $d_{k}=F(k,d)$.
}
\citet{batotr89} note that there are only $m!$ orderings of the
alternatives and thus no more than $m!$ vote types. However, since
we never swap $d$ down the profile \citep{Da06}, the ordering of the candidates
below $d$ are irrelevant. We will formalise this notion as $d$-equivalence:

\begin{defn}
Where $\linord$ is a linear order on $m$ alternatives, let $\linord_{i}$
represent the $i$\tth highest ranked alternative and $\linord_{\leq i}$
represent the sequence of $i$\tth highest ranked alternative\emph{s}.
Where $d$ is an alternative, we say $\linord$ and $\altlinord$
are $d$-equivalent ($\linord\equiv_{d}\altlinord$) iff there exists
$i$ such that $\linord_{i}=d$ and $\linord_{\leq i}=\altlinord_{\leq i}$.
\end{defn}
\begin{lem}
Let $S_{d}$ be the set of $d$-equivalence classes. Then $\left|S_{d}\right|$
is less than $\left(m-1\right)!e$ where $e=2.71\ldots$ is the exponential
constant. 
\end{lem}
\begin{proof}
We see that there is one equivalence class where $d$ is ranked in
the top position, $m-1$ equivalence classes where $d$ is ranked
in the second highest position, and in general $\prod_{k=m-i+1}^{m-1}k$
when $d$ is ranked $i\tth$ from the top. We note that:
\begin{align*}
\prod_{k=m-i+1}^{m-1}k  = & \frac{(m-1)!}{(m-i)!}\end{align*}
We see that \begin{align*}
\left|S_{d}\right|  = & \frac{(m-1)!}{(m-m)!}+\frac{(m-1)!}{(m-(m-1))!}+\cdots+\frac{(m-1)!}{(m-1)!}\\
\iflong{  = & \left(m-1\right)!\left(\frac{1}{0!}+\frac{1}{1!}+\cdots+\frac{1}{(m-1)!}\right)\\}
  < & \left(m-1\right)!\left(\frac{1}{0!}+\frac{1}{1!}+\cdots\right)
  =  \left(m-1\right)!e\tag*{\qedhere}\end{align*}
\end{proof}
\begin{cor}
If we categorise votes into type based on $d$-equivalence classes (instead of linear orders),
the ILP %
below
 has less than $m\left(m-1\right)!e=m!e$
variables.
\end{cor}
Note that there are less than $(m-1)!e$ choices for $i$, no more
than $m$ choices for $j$ and thus less than $m\left(m-1\right)!e=m!e$
variables (each of the form $y_{ij}$).

\begin{lem}
We can transform the ILP \ifshort{of \citet{batotr89}}
 into the following form \ifshort{\citep{Da08Cfull,Da06}}\iflong{\citep{Da06}}:\\
$\min\sum_{i}\sum_{j>0}y_{ij}$ subject to\\
$y_{i0}=N_{i}$ (for each type of vote $i$)\\
$\sum_{ij}(e_{ijk}-e_{i(j-1)k})y_{ij}\geq D_k$ (for each alternative
$k$)\\
$y_{ij}\leq y_{i(j-1)}$ (for each $i$ and $j>0$)\\
$y_{ij}\geq0$, and each $y_{ij}$ must be integer.
\end{lem}
\begin{proof}
For each $i$ and $j$ variable $y_{ij}$ represents the number of
times that the candidate $d$ is swapped up \emph{at least }$j$ positions. In the LP $D_k$ may be defined as 
$\adv(k,d)/2$ to compute the DC score or $\ceil{\adv(k,d)/2}$ to compute the DR score (under the ILP these are equivalent).
\iflong {
Thus, $x_{ij}=y_{i(j-1)}$ and $y_{ij}=\sum_{k\geq j}x_{ik}$. Using
these two equalities it is easy to derive the above ILP.
}
\end{proof}
\begin{thm}
\label{thm:ScR-bound}The DR ($\ScR$), DC ($\ScC$) and D\& ($\ScRR$)
scores are bounded as follows:
\begin{align*}
\ScD(d)-(m-1)!(m-1)e<\ScC(d)\leq\ScR(d)\leq\ScRR\leq\ScD(d)\end{align*}
\end{thm}

\begin{proof}
Every solution to the Integer Linear Program for the Dodgson score
is a solution to the Linear Program for the DR score. Every solution to the LP for the DR score is a solution
to the LP for the DC score. Thus the DC score cannot be greater than
the DR score, which cannot be greater than the Dodgson score ($\ScC(d)\leq\ScR(d)\leq\ScD(d)$).
Since $\ScD$ is integer, it follows that $\ScR(d)\leq\ceil{\ScR(d)}=\ScRR(d)\leq\ScD(d)$.
Also note that given a solution $y$ to either LP, we can produce
a solution $y'$ to the ILP simply by rounding up each variable ($y'_{ij}=\ceil{y_{ij}}$),
hence\begin{align*}
\ScD(d)-\ScC(d)\leq\sum_{i}\sum_{j>0}\ceil{y_{ij}}-y_{ij}<\sum_{i}\sum_{j>0}1\leq\left(m-1\right)!e\end{align*}

Since $i$ can take less than $\left(m-1\right)!e$ values, and $j$
can vary from $1$ to $(m-1)$, it follows that \begin{align*}
\ScD(d)-(m-1)!(m-1)e<\ScC(d)\leq\ScR(d)\leq\ScRR\leq\ScD(d)\tag*{\qedhere}\end{align*}
\end{proof}
\ifshort{
These results can also be used to find tighter bounds on the complexity of solving the ILP and LPs \citep{Da06,Da08Cfull}.
}

\iflong{We can transform the (I)LP into a form with less than $M=(m-1)!e$
variables and encoded in $L\in\bigO\left[\left(m-1\right)!\ln\left(\left(m-1\right)!n\right)\right]$
bits \citep{Da06}.  Using \citeauthor{go89}'s algorithm
we may solve LPs in $\bigO(M^{3}L)$ arithmetic operations of $\bigO(L)$
bits of precision. From these two facts we get the following corollary:}

\iflong{
\begin{cor}
We can compute the DC, DR and D\& scores with\begin{align*}
\bigO\left[\left(\left(m-1\right)!\right)^{4}\ln\left(\left(m-1\right)!n\right)\right]\end{align*}
 operations of\begin{align*}
\bigO\left[\left(m-1\right)!\ln\left(\left(m-1\right)!n\right)\right]\end{align*}
 bits of precision. 
\end{cor}
This suggests that computing these scores for a 100 million agent,
5 alternative profile would be a trivial task for a modern computer.
Hence these rules could be used as a real world approximations to
the Dodgson rule.

We also note that there cannot be more vote types than votes, hence
there are at most $mn$ variables, and the LP can be encoded in $\bigO(mn)$
bits. Hence we can compute these scores in $\bigO(m^{4}n^{4}\ln(mn))$
arithmetic operations requiring $\bigO(mn\ln(mn))$ bits of precision.
Hence computing the DC, DR and D\& scores requires only polynomial
time.

\subsection{Empirical Results.}

In practice the major difference between DR and Dodgson is that the
DR rule picks a smaller set of tied winners. We performed 43 million
simulations \citep{Da06} with up to 25 alternatives and 85 candidates
and assuming various amounts of homogeneity, modeled using the PE
distribution, in the population. We found only 25 cases where DR picked
a different set of tied winners to Dodgson, and only two cases where
the set of tied DR winners was not a subset of the Dodgson winners.
As the DR scores are fractions, the DR rule can split ties according
to which are fractionally better, which may be considered to be more
democratic than e.g.\  splitting ties according to the preferences
of the first agent. In this sense DR is superior to the Dodgson rule.
By rounding the DR scores up, we get a new rule that approximates
Dodgson yet more closely, picking the Dodgson winner in all 43 million
simulations except for a single profile with 25 alternatives and 5
agents generated under the Impartial Culture assumption. 10,000 such
profiles were generated, and by comparison, both the Tideman and DQ
approximations differed from Dodgson in 13\% of these profiles.

The similarity between DR and Dodgson suggests a reason why we were
able to compute the Dodgson winners for non-trivial populations in
milliseconds. LPs tend to be easy to solve, the NP-hardness of ILPs
comes from the difficulty in finding integer solutions. Since the
LP solutions tend to be very close to the integer solutions, finding
the integer solution from the LP solution is often trivial. 

}

\section{\label{sec:Counting-Proof-that}Counting Proof of Convergence under
IAC}

For a voting situation $\votsit$ and linear order $\linord$, we
represent the number of linear orders of type $\linord$ in $\votsit$
by $\votcount{\votsit}{\alinord}$. 

\newcommand{\RuleDiff}[1]{\Delta_{\mathbf{#1}}}
\newcommand{\DT}{\RuleDiff{T}}
\newcommand{\DX}{\RuleDiff{X}}
\newcommand{\DD}{\RuleDiff{D}}
\newcommand{\RuleFont}[1]{#1}
\newcommand{\ScX}{\rulescore{X}}

Where $\RuleFont{X}\in\left\{ \RuleFont{D},\RuleFont{T}\right\} $,
let $\DX\left(a,z\right)$ be equivalent to $\ScX(a)-\ScX(z)$. Given
an arbitrary pair of alternatives $(a,z)$ we pick an arbitrary linear
order $ab\ldots z$ with $a$ ranked first and $z$ ranked last and
call it $\linord$. We also define the reverse linear order $\antiordv=z\ldots ba$.

\begin{lem}
\label{lem:inc-DT}Replacing a vote of type $\antiordv$ with a vote
of type $\linord$ will increase $\DT(a,z)$ by at least one.
\end{lem}
\begin{proof}
We see that replacing a vote of type $\antiordv$ with a vote of
type $\linord$ will increase $\adv(a,z)$ by one, or decrease $\adv(z,a)$
by one. 
\end{proof}
\begin{lem}
\label{lem:inc-DD}Replacing a vote of type $\antiordv$ with a vote
of type $\linord$ will increase $\DD(a,z)$ by at least one.
\end{lem}
\begin{proof}
Say $a$ is not a Condorcet-tie winner, but is a Condorcet-tie winner
after some minimal set $S$ of swaps is applied to the profile $P$.
Let $P'$ be the profile $P$ after one vote of type $\antiordv$
has been replaced with a vote of type $\linord$. If any swaps were
applied to the vote those swaps are no longer required, and $\ScD[P'](a)<\ScD[P](a)$.
Otherwise we can apply the set of swaps $S$ to $P'$ resulting in
$\adv(a,k)$ being at least 2 for all other alternatives $k$. Hence
we can remove one of the swaps, and still result in $a$ being a Condorcet-tie
winner after the swaps have been applied to $P'$.

Say $a$ is a Condorcet-tie winner in $P$. Then $z$ is not a Condorcet-tie
winner in $P'$. As in the previous paragraph we can conclude that $\ScD[P](z)<\ScD[P'](z)$.
\end{proof}
\begin{lem}
\label{lem:For-a-fixed}For a fixed integer $k$, and a fixed ordered
pair of alternatives $(a,z)$ the proportion of voting situations,
with $n$ agents and $m$ alternatives, for which $\DX(a,z)=k$ is
no more than: \begin{align*}
\frac{(m!-2)}{(n+m!-2)}\end{align*}

\end{lem}
\begin{proof}
Say $\linord=ab\ldots z$ is some fixed linear order and $\antiordv=z\ldots ba$
is the reverse order. We define an equivalence relation $\sim$ on
the set of voting situations, as follows: say $\votsit,\altvotsit$
are two voting situations, then \begin{align*}
\votsit\sim\altvotsit\iff\forall_{\altlinord\ne\linord,\,\altlinord\ne\antiordv}\votcount{\altvotsit}{\altlinord}=\votcount{\votsit}{\altlinord}.\end{align*}
From \prettyref{lem:inc-DT} and \ref{lem:inc-DD} we see that in
each equivalence class, there can be at most one voting situation
for which $\DX(a,z)=k$ . Also note that whereas there are\begin{align*}
|\SSS^{n}(A)|= & {n+m!-1 \choose n}\end{align*}
distinct voting situations there are at most \begin{align*}
{n+(m!-1)-1 \choose n}\end{align*}
equivalence classes under $\sim$. Hence the proportion of voting
situations for which $\DX(a,z)=k$ is no more than:\begin{align*}
\frac{(n+m!-1)!}{n!(m!-1)!}\frac{n!(m!-2)!}{(n+m!-2)!}= & \frac{(n+m!-1)!}{(n+m!-2)!}\frac{(m!-2)!}{(m!-1)!}=\frac{(m!-2)}{(n+m!-2)}\tag*\qedhere\end{align*}

\end{proof}
\begin{lem}
If $\DT(a,z)=k$ and $a$ is a Tideman winner and $z$ is a DQ winner,
then $0\leq k<m$.
\end{lem}
\begin{proof}
 As $a$ is a Tideman winner and $z$ is a DQ winner, then\begin{align*}
\ScT(a)\leq\ScT(z),\quad\ScQ(z)\leq\ScQ(a)\end{align*}
Recall that \begin{align*}
\ScQ(x)=\sum_{y\ne x}\left\lceil \frac{\text{adv}(y,x)}{2}\right\rceil ,   & \quad\ScT(x)=\sum_{y\ne x}\text{adv}(y,x)\textrm{.}\end{align*}
We see that $\text{adv}(y,x)\leq2\left\lceil \text{adv}(y,x)/2\right\rceil \leq\text{adv}(y,x)+1$
and so $\ScT(x)\leq2\ScQ(x)<\ScT(x)+m$ for all alternatives $x$.
Thus\begin{align*}
\ScT(z)\leq2\ScQ(z)\leq2\ScQ(a)<\ScT(a)+m\textrm{.}\end{align*}
And so $\ScT(a)\leq\ScT(z)<\ScT(a)+m$. Let $k=\ScT(a)-\ScT(z)$.
Then $0\leq k<m$, and so there are no more than $m$ ways of choosing
$k$ if we wish the DQ and Tideman winners to differ.
\end{proof}
Recall that \prettyref{thm:ScR-bound} states:\begin{align*}
\ScD(d)-(m-1)!(m-1)e<\ScC(d)\leq\ScR(d)\leq\ScRR\leq\ScD(d)\end{align*}
where $e=2.71\ldots$ is the exponential constant. Given that there
are only $m$ ways of choosing $k$ such that the Tideman
and DQ winners differ, and less than $(m-1)!(m-1)e$ ways of choosing
$k$ such that the Dodgson, DC, DR and/or D\& winners differ, we get
the following theorem.

\begin{thm}
The proportion of voting situations, with $n$ agents and $m$ alternatives,
for which $a$ is a Tideman winner and $z$ is a DQ winner is no
more than:\begin{align*}
\frac{(m!-2)}{(n+m!-2)}m.\end{align*}
The proportion for which $a$ is Dodgson winner and $z$ is a DC,
DR and/or D\& winner is less than\textup{\begin{align*}
\frac{(m!-2)}{(n+m!-2)}(m-1)!(m-1)e.\end{align*}
}
\end{thm}
\iflong{
\begin{thm}
The proportion of voting situations, with $n$ agents and $m$ alternatives,
for which the Tideman winner differs from the DQ winner is no more
than:\begin{align*}
\frac{(m!-2)}{(n+m!-2)}(m-1)m^{2}.\end{align*}
The proportion where the Dodgson winner differs from the DC, DR, and/or
D\& winners is less than\begin{align*}
\frac{(m!-2)}{(n+m!-2)}m!(m-1)^{2}e.\end{align*}
\end{thm}
\begin{proof}
Obvious, as there are only $m(m-1)$ ways of choosing $a$ and $z$
from the set of alternatives. 
\end{proof}
}
\ifshort{ 
As there are only $m(m-1)$ ways of choosing $a$ and $z$ from the set of alternatives, we get the following corollary 
}
\begin{cor}
The probability that the DQ and Tideman rule pick the same winners
converges to $1$ as $n\rightarrow\infty$, under the Impartial Anonymous
Culture assumption. Likewise the probability that the DR, DC, D\&
and Dodgson rules pick the same winner converges to 1 as $n\rightarrow\infty$.
\end{cor}
In other words the DR, DC and D\& winners (and scores) provide {}``deterministic
heuristic polynomial time'' \citep{PrRo07} algorithms for the
Dodgson winner (and score) with the IAC distribution.%
We can use the same technique to show that the Greedy Algorithm proposed
by \citet{homan-2005-} converges to the DQ and Tideman rules under
IAC, as the \texttt{GreedyScore} differs from the DQ score by less than $m$.%
By setting $k$ to 0 we may likewise prove that the probability of
a non-unique Tideman or Dodgson winner converges to 0 under IAC\@.%

We extend the concept of a {}``frequently self-knowingly correct
algorithm'' \citep{homan-2005-} such that we can specify a distribution
over which the algorithm is frequently self-knowingly correct.
\begin{defn}

A self-knowingly correct \citep{homan-2005-} algorithm $A$ is a {}``frequently self-knowingly
correct algorithm over a
distribution $\mu$''  for $g\colon\Sigma^{*}\rightarrow T$  iff\begin{align*}
\lim_{n\rightarrow\infty} & \sum_{x\in\Sigma^{n},\, A(x)\in T\times\{\textrm{maybe}\}}P_{\mu}(x)=0\end{align*}
where $P_{\mu}(x)$ is the probability that $X=x$ when $X$ is chosen
from $\Sigma^{n}$ under the $\mu$ distribution and for all $x$ we have $(A(x))_1=g(x)$ or $(A(x))_2=\textrm{``maybe''}$.
\end{defn}

Using the set of linear orders $\allorders$ as $\Sigma$ we may construct
a frequently self-knowingly correct algorithm from DC (or DR, or D\&)
as the algorithm can output {}``definitely'' whenever the DC winner
has DC score that is at least $(m-1)!(m-1)e$ less than any other
alternative.

\section{\label{sec:Non-convergence-of-Tideman}Non-convergence of Tideman
Based rules}

\makeatletter

   \newcommand{\firstcap}[1]{\@firstcap #1}%

   \newcommand{\@firstcap}[1]{\uppercase{#1}}%

\renewcommand{\underbar}[1]{\textbf{#1}\index{#1@\firstcap{#1}}}

\renewcommand{\term}[1]{\textbf{#1}\index{#1@\firstcap{#1}}}

\ifpolya{
Under the P\'olya-Eggenberger (PE) urn model we start with a non-negative
integer $a$ and an urn containing balls, each of a different colour.
To generate each random sample, we pull a ball out of the Urn at random
and note its colour. After removing each ball, we return the ball that
was taken to the urn together with $a$ additional balls of the same
colour to the urn. Our sample is the ball taken from the urn. 

If $a=0$,  PE generates a binomial or multinomial distribution and
these distributions converge to the standard distribution as the number
of samples increase to infinity. As the number of samples increase
to infinity, it is almost certain that the ratio of the balls drawn
will match the ratio of the balls in the urn. 

If $a>0$ the ratio of balls in the urn varies. We see that for any
 $x\in(0,1)$ and $\epsilon>0$ there is a finite set of draws which
will result in the ratio of balls of in the urn ratio being within
the range $(x-\epsilon,x+\epsilon)$. Because of this, the samples
drawn from the urn may converge on any ratio as the number of balls
drawn approaches infinity. More formally, the PE distribution converges
to the Beta distribution as the number of samples approaches infinity
\citep{Fe71}, thus the limiting probability of the ratio falling
within any open set is non-zero. 

For generation of random profiles, we replace colours with linear orders.
The parameter $a$ characterises homogeneity; for $a=0$ we obtain
the well known Impartial Culture assumption and for $a=1$ the 
Impartial Anonymous Culture assumption \citep{bergle94}. 
}

\begin{defn}
A {}``voting ratio'' is a function $f\colon\LL(\allalts)\rightarrow[0,1]$
such that\begin{align*}
\sum_{\vctr{v}\in\LL(\allalts)}f(\vctr{v}) & =1\textrm{.}\end{align*}
We say that a profile $\profile$ reduces to a voting ratio $f$ if\begin{align*}
\forall_{\vctr{v}\in\LL(\allalts)}\#\{i:\,\profile_{i}=\vctr{v}\}  = & nf(\vctr{v})\end{align*}

\end{defn}
\begin{note}
A voting ratio is similar to a voting situation, but unlike a voting
situation does not contain any information about the total number
of agents. 
\end{note}
\begin{defn}
We say that a voting ratio $f$ is {}``bad'' if for every profile
$\profile$ that reduces to $f$ and has an even number of agents,
the DQ winner of $\profile$ differs from the Dodgson winner.
\end{defn}
\iflong{
\begin{example}
The following voting ratio is bad.\begin{align*}
g(\vctr{v})  = & \left\{ \begin{array}{ccc}
\jnicefrac{7}{18}  \textrm{ if } & \vctr{v}=abcde\\
\jnicefrac{6}{18}  \textrm{ if } & \vctr{v}=cdabe\\
\jnicefrac{5}{18}  \textrm{ if } & \vctr{v}=bcead\\
0   & \textrm{otherwise}\end{array}\right.\end{align*}
For any profile with $18n$ agents that reduces to the above voting
ratio, the DQ and Dodgson score of $c$ will be $3n$; the DQ score
of $a$ will be $2n$ and the Dodgson score of $a$ will be $4n$.
Hence $a$ will be the DQ winner but $c$ will be the Dodgson winner.
\end{example}
}

\begin{example}\label{exa:h}
The following voting ratio is bad.\begin{align*}
h(\vctr{v})  = & \left\{ \begin{array}{ccc}
\jnicefrac{16}{39}  \textrm{ if } & \vctr{v}=abcx\\
\jnicefrac{12}{39}  \textrm{ if } & \vctr{v}=cxab\\
\jnicefrac{10}{39}  \textrm{ if } & \vctr{v}=bcxa\\
\jnicefrac{1}{39}  \textrm{ if } & \vctr{v}=cbax\\
0   & \textrm{otherwise}\end{array}\right.\end{align*}
We see that for a profile that reduces to the above voting ratio,
we have the following advantages and scores per 78 agents:
\end{example}
<<<<<<< .mine
\vspace{0.25cm}
=======
\vspace{0.5cm}
>>>>>>> .r176
\begin{center}
\begin{minipage}{5cm}\vspace{-3cm}\rotatebox{-90}{\includegraphics[scale=0.7]{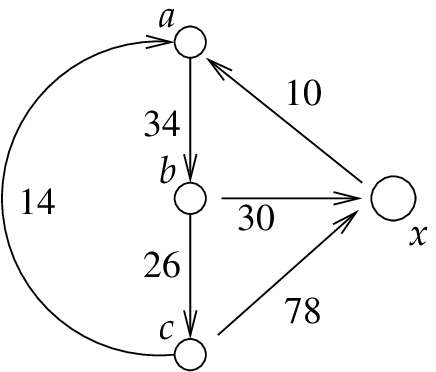}}\end{minipage}\raisebox{1.5cm}{\begin{tabular}{|c|c|c|c|c|}
\hline 
 & $a$ & $b$ & $c$ & $x$\tabularnewline
\hline
\hline 
DQ score & 12 & 17 & 13 & 54\tabularnewline
\hline 
Dodgson score & 14 & 17 & 13 & 54\tabularnewline
\hline
\end{tabular}}
<<<<<<< .mine
\vspace{-0.35cm}
=======
>>>>>>> .r176
\par\end{center}

The Dodgson score of $a$ is higher than the DQ score of $a$ as to
swap $a$ over $c$ we must first swap $a$ over $b$ or $a$ over
$x$. We have to swap $a$ over $c$ at least $7$ times. Hence we
must use a total of at least 14 swaps to make $a$ a Condorcet-tie
winner. We see that $a$ is the DQ winner, but $c$ is the Dodgson
winner.

We will now show that there exists a neighbourhood around the voting
ratios $g$ and $h$ that is bad.

\begin{lem}
Altering a single vote will change the Dodgson scores and DQ scores by at most $m-1$. 
\end{lem}
\begin{proof} Recall that
\begin{align*}
\ScQ(a)  = & \sum_{b\ne a}\fQ(b,a),\text{ where }\fQ(b,a)=\left\lceil \frac{\text{adv}(b,a)}{2}\right\rceil .\end{align*}
We see that for each other alternative $b$ changing a single vote can change $\fQ(b,a)$ by at most one, and there are $m-1$ such alternatives.

Say that $P$ and $R$ are two profiles that differ only in a single vote.
Let $P'$ and $R'$ be $P$ and $R$ respectively after some arbitrary $d$ has been swapped to the top of the vote that differs, which requires no more than $m-1$ swaps. We see that
\begin{align*}
\ScD[P](d) \leq \ScD[P'](d)+m-1 \leq \ScD[R](d)+m-1  & \\ 
\ScD[R](d) \leq \ScD[R'](d)+m-1 \leq \ScD[P](d)+m-1  & \textrm{ .}  \tag*{\qedhere}
\end{align*}  
\hspace{-1em}
\end{proof}

\begin{cor}
For any positive integer $k$, alternative $d$, profile $P$
and rule $X\in\{\RuleFont{D},\RuleFont{Q}\}$ (i.e. Dodgson or DQ),
 if $\ScX(d)<\ScX(a)-2k(m-1)$ for all other alternatives $a$,
then $d$ will remain the unique $X$ winner in any profile that results from changing $k$ or less votes.
\end{cor}

\begin{lem}
There is a neighbourhood of bad voting ratios around the voting ratio $h$ from Example \ref{exa:h}.
\end{lem}
\begin{proof}
We see that in a profile with $n$ agents that reduces to the voting ratio
$h$ the Dodgson and DQ winners have scores that are at least $\frac{n}{78}$ lower than the other alternatives. Thus if we alter less than $\frac{n}{78}\frac{1}{2(1)(4-1)}=\frac{n}{468}$ votes, the DQ winner will remain different from the Dodgson winner.
\end{proof}

\ifpolya{
\begin{thm}
If we generate profiles randomly according to the PE distribution,
with $a>0$ and  $abcx$, $cxab$, $bcxa$, $cbad$ existing in the
urn, the probability that the DQ winner is not the Dodgson winner
will converge to some non-zero value as the number of agents tends
to infinity.
\end{thm}
\begin{proof}
The PE distribution converges to the Beta distribution. Given any
neighbourhood, the probability that a variable distributed according
to a non-trivial beta distribution will fall within that neighbourhood
is non-zero. There is a neighbourhood of bad voting ratios and any
profile that reduces to a bad voting ratio has a DQ winner that is
different to the Dodgson winner.
\end{proof}

We can generalise this result
by adding alternatives that lose to all of $a$, $b$, $c$ and $x$.
As the Impartial Anonymous Culture assumption is a special case of
the PE distribution with $a=1$ and the urn starting with an equal
number of all linear orders, these approximations also do not converge
under the IAC assumption.} Thus we may now state the following theorem:

\begin{thm}
If we generate profiles randomly according to the IAC distribution,
with $m\geq4$, the DQ, Greedy and Tideman winners do not converge
to the Dodgson winner as the number of agents tends to infinity.
\end{thm}

\ifnotpolya{It is easy to generalise this result to non-trivial
P\'olya-Eggenberger distributions. See \citep{Da06} for details.}

Under the IAC, the Tideman rule \citep{DaPrSl07} and
the greedy algorithm proposed by \citet{homan-2005-} converges to the DQ rule, which does not converge to Dodgson's rule. It follows that the Tideman rule and greedy algorithm do not converge to Dodgson's rule under the IAC.
As the difference between the DQ scores and the \texttt{GreedyScores}
\citep{homan-2005-} is less than $m$, the DQ winner and \texttt{GreedyWinner}
<<<<<<< .mine
will converge under \iflong{the} IAC.
=======
will converge under the IAC.
>>>>>>> .r176

\section{Conclusion}

We have previously \citep{DaPrSl07} proved that Tideman's rule does converge under IC, and presented a refined rule (DQ) which converged exponentially quickly. Another approximation with exponentially fast
convergence was independently proposed by \citet{homan-2005-}.

Unfortunately, real voters are not independent, and so the Impartial
Culture assumption is not realistic. This paper investigates the asymptotic
behaviour of approximations to the Dodgson rule under other assumptions
of voting behaviour, particularly the Impartial Anonymous Culture
(IAC) assumption, that each multiset of votes is equally likely.

We found that the approximations converge to each other under IAC,
but they do not converge to Dodgson. Hence the DQ rule is not asymptotically
closer to the Dodgson rule than the Tideman rule is, although DQ converges
faster under the IC. 

It is not realistic to assume that voters' behaviour will precisely
follow any mathematical model, including IAC\@. Fortunately the proof
that our new approximations converge to Dodgson does not depend on
the details of the IAC\@. Indeed, the DC, DR, D\& and Dodgson winners
will all be the same if the  scores of the two leading alternatives
differ by less than $(m-1)!(m-1)e$. This assumption is realistic for a sufficiently
large number of voters. Even in the neck and neck 2000 US presidential
elections, the popular vote for the leading two alternatives differed
by half a percent --- over half a million votes. If we used the Dodgson
rule to choose between the top five alternatives, the difference between
the Dodgson scores of the leading two alternatives would have to be
less than 261 for the winners to differ. Even if the entire electorate
conspired to cause the winners to differ, minor inaccuracies such
as hanging chads could frustrate their attempt. Even in cases where
the Dodgson Relaxed does differ from the Dodgson, the difference seems
to be primarily that the Dodgson Relaxed rule picks a smaller set
of tied winners as it is able to split ties in favour of alternatives
that have fractionally better DR scores. This is in some sense more
democratic than tie breaking procedures such as breaking ties in favour
of the preferences of the first voter.

Thus determining that an algorithm is {}``frequently self-knowingly
correct'', as defined by \citet{homan-2005-}, is insufficient to
conclude that the algorithm will converge in practice. The DQ and
\texttt{GreedyWinner} provide frequently self-knowingly correct algorithms, but do
not converge under other assumptions of voter behaviour such as IAC.
We have extended the definition to allow a distribution to be specified.
So, in other words, DQ and \texttt{GreedyWinner} provide algorithms that are {}``frequently
self-knowingly correct over IC'', but unlike DC, DR, and
D\&, do not provide algorithms that are ``frequently self-knowingly correct algorithms over 
IAC''.

We have previously shown that the Dodgson scores and winners of a
voting situation can be computed from a voting situation with $\bigO(f_{1}(m)\ln n)$
arithmetic operations of $\bigO(f_{2}(m)\ln n)$ bits of precision
\citep{Da06} for some pair of functions $f_{1}$ and $f_{2}$.
However we did not find a good upper bound on $f_{1}$ or $f_{2}$,
even $f_{1}(4)$ may be unreasonably large. \iflong{In contrast, we found
that we may compute the DC, DR and D\& score with $\bigO\left[\left(\left(m-1\right)!\right)^{4}\ln\left(\left(m-1\right)!n\right)\right]$
operations of\begin{align*}
\bigO\left[\left(m-1\right)!\ln\left(\left(m-1\right)!n\right)\right]\end{align*}
 bits of precision. Although $\left(\left(m-1\right)!\right)^{4}$
is quite large, a modern computer can perform billions of elementary
operations a second. This result suggests that even the worst case
time for billions of voters and 6 or 7 alternatives is feasible. Also,
whereas \texttt{DodgsonScore} is NP-complete \citep{batotr89}, we can also compute the
DR and D\& scores using $\bigO(m^{4}n^{4}\ln(mn))$ arithmetic operations
requiring $\bigO(mn\ln(mn))$ bits of precision. } 
\ifshort{This suggests that it may be important to use a variant of Dodgson's rule that is truely polynomial, such as DC, DR or D\&. See \citep{Da06} for a discussion of the complexity of these rules.}%

We have found that the scores of the most of the approximations we
have studied form a hierarchy of increasingly tight lower bounds on
the Dodgson score:\begin{align*}
\frac{\ScS(x)}{2} & \leq\frac{\ScT(x)}{2}\leq\ScQ\leq\ScR\leq\ScRR\leq\ScD\leq\ScR+(m-1)!(m-1)e\end{align*}
The Dodgson Clone rule does not fit in that hierarchy, although it is the case that 
\begin{align*}
\frac{\ScT(x)}{2}\leq\ScC\leq\ScR\leq\ScD\leq\ScC+(m-1)!(m-1)e \textrm{ .} &
\end{align*}
Despite the great accuracy of the D\& approximation, there are good
reasons to pick other approximations. The difference between the DR
rule and the D\& rule is that the DR rule can split ties based on
fractional scores, so the DR rule may be considered superior to D\&
and Dodgson's rule. The DC rule is resistant to cloning of the electorate.
The DQ rule is very simple to compute, and very
easy to write in any programming language. As the DQ rule is known
to converge exponentially fast under IC, this makes the DQ rule very
appropriate for cases where the data is known to distributed according
to IC\@. This is the case for studies that have randomly generated
data according to the IC \citep[see e.g.][]{explore,shah1,nu83}.
The Tideman rule is no more easy to compute than the DQ rule. However,
the mathematical definition of the Tideman rule is simpler than the
DQ rule. This makes the Tideman rule useful for theoretical studies
of the Dodgson rule where the speed of convergence is not important.
Also, like the DC rule, the Tideman rule is resistant to cloning the
electorate.

We conclude that the DC and DR rules are superior,
for social choice, to the traditional definition of the Dodgson rule. 
DC provies resistance to cloning the electorate, and both are better at splitting ties than the traditional definition of Dodgson rule.
We know of no advantage of the traditional definition over these rules.
Even if the traditional Dodgson winner is preferred, it may be hard to justify the computational complexity
of the traditional Dodgson rule, especially since it is almost certain
that these rules would pick the same result. If, despite all this,
the traditional Dodgson rule is still chosen, these new rules provide
frequently self-knowingly correct algorithms for the Dodgson winner
for any reasonable assumption of voter behaviour, including IAC.

\bibliographystyle{johnnat}
\bibliography{mybibtex}

\begin{contact}
John Christopher M\textsuperscript{c}Cabe-Dansted\\
M002 The University of Western Australia\\
35 Stirling Highway, Crawley 6009 \\
Western Australia\\
\email{john@csse.uwa.edu.au}
\end{contact}

\end{document}